\documentclass{article}
\usepackage{amsmath}
\usepackage{latexsym}
\usepackage{amsfonts} 
\usepackage{amsthm}
\newtheorem{lemma}{Lemma}
\newtheorem{theorem}{Theorem}
\newtheorem{corollary}{Corollary}

\def\E{\mathbb{E}}
\def\P{\mathbb{P}}
\def\I#1{{\mathbf{1}}_{#1}}
\title{An Inequality for the Sum of Independent Bounded Random Variables}
\author{Christopher R. Dance \\
{\small Xerox Research Centre Europe, 6 chemin de Maupertuis, 38240 Meylan, France}\\
{\small {{\tt dance@xrce.xerox.com}\ Phone: +33 4 76 61 51 37}}
}
\date{September 2012}
\begin{document}
\maketitle
\abstract{We give a simple inequality for the sum of independent bounded random variables.
This inequality improves on the celebrated result of Hoeffding in a special case.
It is optimal in the limit where the sum tends to a Poisson random variable.}

\section{Introduction}
Modern machine learning and stochastic programming are largely based on inequalities relating to the sums of random variables.
Hoeffding~\cite{Hoeffding63} proposed several such bounds, 
which were in turn improved by Talagrand~\cite{Talagrand95}, Pinelis~\cite{Pinelis98} and Bentkus~\cite{Bentkus04}.
In this paper we prove the following related result.

\begin{theorem}
\label{theorem1}
Suppose that $S = \sum_{i=1}^n X_i$ is a sum of independent random variables with
$\mathbb{P}(0\leq X_i \leq 1)=1$ for $1\leq i\leq n$ and $\E S = \lambda$.
Then 
\begin{align}
\label{theorem1_1}
\P(S\leq 1) &\leq \max\left\{ \left(1 + \lambda -\frac{\lambda}{n}\right)\left( 1 - \frac{\lambda}{n} \right)^{n-1}, 
\left( 1 - \frac{\lambda - 1}{n-1} \right)^{n-1} \right\} \\
\label{theorem1_2}
&\le \max\{ 1 + \lambda, e \} e^{-\lambda}.
\end{align}
\end{theorem}

In this context, Hoeffding's inequality states that for all $\lambda \ge 1$
\begin{align}
\P(S\le 1) \le \lambda \left( 1 + \frac{1-\lambda}{n}\right)^{n-1}.
\end{align}
Theorem~\ref{theorem1} is not as general as Hoeffding's inequality since it only allows us to bound $\P(S\leq 1)$
rather than $\P(S\leq t)$ for any positive $t$.
However, from Theorem~\ref{theorem1}, we may derive Corollary~\ref{Corollary1} which states that
\begin{align}
\mathbb{P}(S\leq 1) \leq e^{1 - r\E S} \text{ where } r = 0.841405\dots .
\end{align}
In contrast, the strongest such result that can be obtained from Hoeffding's bound is 
\begin{align}
\mathbb{P}(S \leq 1) \leq e^{1 - \left(1 - e^{-1}\right) \E S} \text{ where } 1 - e^{-1} = 0.6321\dots .
\end{align}
Thus Theorem~\ref{theorem1} improves on the Hoeffding bound.

It is interesting to compare our result with Theorem~1.2 of Bentkus~\cite{Bentkus04} in the form of his inequality~1.1.
This states that for a sequence of bounded independent random variables $Y_i$ 
such that $\P(0\leq Y_i \leq 1)$ we have
\begin{align}
\P\left( \sum_{i=1}^n Y_i \geq x \right) \leq e\P(B_n \geq x)
\end{align}
where $B_n \sim \mathrm{binomial}(p, n)$ with 
$p := \sum_{i=1}^n \E Y_i / n$.
If we set $X_i := 1-Y_i$ and $S =\sum_{i=1}^n (1 - Y_i)$  in order to match the random variables in our Theorem~\ref{theorem1}, 
Bentkus's result gives 
\begin{align*}
\P(S \leq 1) &= \P\left( \sum_{i=1}^n Y_i \geq n - 1\right) \leq e \left( p^n + n (1-p) p^{n-1}\right) .
\end{align*}
If we set $m := \E S$, so that $p = 1 - \frac{m}{n}$, 
we have $p^n = \left( 1 - \frac{m}{n} \right)^n \leq e^{-m}$ and $p + n(1-p) \leq 1+m$ so that
\begin{align*}
\P(S \leq 1) \leq \frac{e}{p} (1+\E S ) e^{-\E S}.
\end{align*}
This bound is a factor of $e$ larger than our result for all $\E S \geq e-1$.

Furthermore, Theorem~\ref{theorem1} is optimal in the following sense. 
If 
\begin{align*}
S\sim \mathrm{binomial}\left(\frac{\lambda}{n}, n\right)
\end{align*}
then
\begin{align*}
\P(S\leq 1) = \mathbb{P}(S = 0) + \mathbb{P}(S=1) = 
\left( 1 - \frac{\lambda}{n} \right)^{n} + \lambda \left( 1 - \frac{\lambda}{n} \right)^{n-1}
\end{align*}
corresponding to the first term in the `$\max$' of~(\ref{theorem1_1}),
while if 
\begin{align*}
S\sim 1 + \mathrm{binomial}\left(\frac{\lambda-1}{n-1}, n-1\right)
\end{align*}
then we get the second term in the `$\max$'
\begin{align*}
\P(S\leq 1) &= \left( 1 - \frac{\lambda-1}{n-1} \right)^{n-1} .
\end{align*}
Similarly, in the $n$-independent form of our bound~(\ref{theorem1_2}), 
if $S\sim \mathrm{Poisson}(\lambda)$ then
\begin{align*}
\P(S\leq 1) = \mathbb{P}(S = 0) + \mathbb{P}(S=1) = 
e^{-\lambda} + e^{-\lambda} \lambda 
\end{align*}
corresponding to the first term in the `$\max$' of~(\ref{theorem1_2}).
Similarly, if $S\sim 1 + \mathrm{Poisson}(\lambda-1)$ then we get the second term in the `$\max$'
\begin{align*}
\P(S\leq 1) &= e^{1-\lambda} .
\end{align*}
While the sum of a finite collection of bounded random variables $\sum_{i=1}^n X_i$ cannot have a Poisson distribution, 
the law of small numbers implies that the Poisson
distribution is the limit as $n\rightarrow\infty$ of the sum of a suitable collection of random variables $(X_i)_{i=1, 2, \dots, n}$.
For instance if each $X_i$ is a Bernoulli random variable taking value $1$ with probability $\lambda/n$ 
and value $0$ otherwise, then following limiting probability mass function is Poisson
\begin{align}
\lim_{n\rightarrow \infty} \mathbb{P}\left(\sum_{i=1}^n X_i = x\right) = e^{-\lambda} \frac{\lambda^x}{x!} \ \text{for} \ x\in \mathbb{Z}_+ .
\end{align}

\section{Proof of Theorem~\ref{theorem1}}
In this section, we define four families of random sums $\mathcal{S}_n, \mathcal{T}_n$, $\mathcal{U}_n$ and $\mathcal{V}_n$.
Then we present Lemmas~\ref{lemmaS}, \ref{lemmaT}, \ref{lemmaU} and~\ref{lemmaV} that relate these families, and combine these results to prove Theorem~\ref{theorem1}.

The random variable considered in Theorem~\ref{theorem1} is from the family
$\mathcal{S}_n$ of random variables $S$ of the form
\begin{align}
\label{SLinea}
S := \sum_{i = 1}^n X_i
\end{align}
where $X_i$ are independent random variables with $X_i\in [0,1]$.
Family $\mathcal{T}_n$ is the set of Bernoulli sums $T$ of the form
\begin{align}
\label{TLinea}
T := \sum_{i = 1}^n Y_i
\end{align}
where $Y_i$ are independent random variables taking values $a_i$ or $b_i$
with $a_i, b_i \in [0,1]$, for $i = 1, 2, \dots, n$.
Family $\mathcal{U}_n$ is the set of Bernoulli sums $U$ of the form
\begin{align}
\label{ULinea}
U := \sum_{i=1}^n B_i
\end{align}
with each Bernoulli random variable $B_i$ 
taking the value 0 or 1.
Finally, family $\mathcal{V}_n$ is the set of shifted binomial random variables $V$ with 
any parameter $p\in [0,1]$ and with either of the following two forms
\begin{align}
\label{VLinea}
V \sim \mathrm{binomial}(p, n) \quad \text{ or } \quad V \sim 1 + \mathrm{binomial}\left(\frac{np-1}{n-1}, n-1\right) .
\end{align}

\begin{lemma}
\label{lemmaS}
For any random sum $S\in\mathcal{S}_n$, there exists
a random sum $T\in\mathcal{T}_n$ such that
$\E S = \E T$ and $\P(S\le 1) \le \P(T\le 1).$
\end{lemma}

Lemma~\ref{lemmaS} follows directly from Theorem~8 of Mulholland and Rogers~\cite{Mulholland58}, 
which we state as Lemma~2.
\begin{lemma}
\label{MulhollandRogers}
For each integer $i$ with $1\le i \le n$, let $f_{i1}(x), \dots, f_{ik}(x)$ 
be Borel-measurable functions, 
and let $K_i$ be the set of probability distribution functions
$F_i(x) : \mathbb{R}\rightarrow [0, 1]$ satisfying
\begin{align}
\label{momentCondition}
\int_{-\infty}^\infty f_{ij}(x) dF_i(x) = 0 \qquad \text{for $j=1, 2, \dots, k$}.
\end{align}
Let $E_i$ be the set of functions from $K_i$
that are step-functions having $\kappa_i$ jumps at points $x_{i1}, x_{i2}, \dots, x_{i\kappa_i}$ where
$1\le \kappa_i\le k+1$ and where the $\kappa_i$ vectors
\begin{align*}
(1, f_{i1}(x_{ij}), \dots, f_{ik}(x_{ij})) \qquad j = 1, 2, \dots, \kappa_i
\end{align*} 
are linearly independent.

Suppose that $g(x_1, x_2, \dots, x_n)$ is Borel-measurable as a function of the point $(x_1, \dots, x_n) \in \mathbb{R}^n$. Then 
\begin{align*}
\sup_{F_i(x) \in K_i} \int_{-\infty}^\infty \dots \int_{-\infty}^\infty g(x_1, \dots, x_n) dF_1(x_1) \dots dF_n(x_n) \\
\quad = 
\sup_{H_i(x) \in E_i} \int_{-\infty}^\infty \dots \int_{-\infty}^\infty g(x_1, \dots, x_n) dH_1(x_1) \dots dH_n(x_n) 
\end{align*}
provided the left-hand side is finite.
\end{lemma}

\begin{proof} See~\cite{Mulholland58}.
\end{proof} 

In the above Lemma, the conditions~(\ref{momentCondition}) can be interpreted as
moment conditions on random variables $X_i$ whose probability distribution functions are $F_i$, 
while the distribution functions in $E_i$ correspond to random variables whose support consists of a finite set of $\kappa_i$ points and which satisfy conditions~(\ref{momentCondition}).

\begin{proof}[Lemma~\ref{lemmaS}]
Let $\I{C}$ denote the indicator function for condition $C$
and let $X_i$ be the random variables defining $S$ for $1\le i \le n$.
In Lemma~\ref{MulhollandRogers}, put
\begin{align}
f_{i1}(x) := \I{0\le x \le 1} - 1, f_{i2}(x) := x - \E X_i
\end{align} 
which are both Borel-measurable functions.
Then the set $E_i$ of distribution functions corresponds to the set of random variables $Z_i$ 
which take on $1 \le \kappa_i \le 3$ distinct values, say at $z_{i1}, \dots, z_{i\kappa_i}$,
which satisfy
\begin{align}
\P(0\le Z_i \le 1) = 1 \text{ and } \E Z_i = \E X_i
\end{align}
and for which the vectors
$(1, f_{i1}(z_{ij}), f_{i2}(z_{ij}))$ are linearly independent for $1\le j \le \kappa_i$.

We now rule out the case $\kappa_i = 3$, since if $\P(0\le Z_i \le 1) = 1$
then the jumps must satisfy $f_{i1}(z_{ij}) = \I{0\le z_{ij} \le 1} -1 = 0$ 
and there are at most two linearly independent
vectors of the form 
\begin{align}
(1, 0, z_{ij} - \E X_i).
\end{align}
Thus the random variables $Z_i$ take on at most two values, say $a_i, b_i \in [0, 1]$, and 
so the random variables $Z_i$ match the definition of the random variables $Y_i$ defining the sum $T$.

Finally, if we set $g(x_1, \dots, x_n) := \I{\sum_{i=1}^n x_i \le 1}$, which is Borel-measurable,
and identify the distribution functions $F_i(x)$ with those of the random variables $X_i$ then Lemma~\ref{MulhollandRogers} gives
\begin{align}
\nonumber
\int_{-\infty}^\infty \dots \int_{-\infty}^\infty g(x_1, \dots, x_n) dF_1(x_1) \dots dF_n(x_n) 
\\
= \P\left(\sum_{i=1}^n X_i \le 1\right) \le \P\left(\sum_{i=1}^n Z_i \le 1\right)
\end{align}
since $\P\left(\sum_{i=1}^n X_i \le 1\right) \in [0,1] .$ 
This completes the proof. 
\end{proof}

\begin{lemma}
\label{lemmaT}
For any $T\in\mathcal{T}_n$ there exists a $U\in \mathcal{U}_1 \cup \mathcal{U}_2 \cup \cdots \cup \mathcal{U}_n$ such that 
\begin{align}
\label{Ucondition}
\P(T\le 1) \le \P(U\le 1) \ \text{ and } \ \E T \le \E U.
\end{align}
\end{lemma}

\begin{proof}
We use induction on $n$.

If $n=1$ then we set $U=1$, so that 
\begin{align*}
\P(T\leq 1) = \P(U\leq 1) = 1 \text{ and } \E T \leq \E U
\end{align*}
directly satisfying~(\ref{Ucondition}).

If $n>1$ there are several cases to consider, for which it helps to first rewrite $T$. 
Recall that $T = \sum_{i=1}^n Y_i$ where $Y_i \in \{a_i, b_i\}$ and $0\le a_i \le b_i \le 1$. 
Thus we may write
\begin{align}
T = a + \sum_{i=1}^n c_i B_i
\end{align}
where $a := \sum_{i=1}^n a_i \geq 0$, $c_i := b_i - a_i \in [0,1]$ 
and where $B_i$ are independent Bernoulli random variables.

In the first case, if $a>1$ then we put $U = \sum_{i=1}^n B_i'$ 
with $\P(B_i' = 1)=1$ for all $i$. Then 
$\P(T \leq 1) = \P(U\le 1) = 0$ and $\E T \le \sum_{i=1}^n b_i \le n = \E U$, satisfying~(\ref{Ucondition}).

Secondly, if $c_i + c_j \leq 1-a$ for some pair $i, j$ with $i\neq j$, then consider the sum 
\begin{align}
S := X_{ij} + \sum_{k\in \{1, 2, \dots, n\}\backslash \{i, j\}} Y_k \text{ where } X_{ij} := Y_i + Y_j.
\end{align}
We have 
\begin{align*}
X_{ij} = a_i + a_j + c_i B_i + c_j B_j \leq a + c_i + c_j \leq 1
\end{align*}
for all realizations of $B_i, B_j$. 
Thus $S\in \mathcal{S}_{n-1}$ and we can apply Lemma~\ref{lemmaS} to show that 
$\E T = \E S \le \E T'$ and $\P(T\le 1) = \P(S\le 1) \le \P(T'\le 1)$ 
for some $T' \in  \mathcal{T}_{n-1}$.
The Lemma then follows by induction.

Otherwise, we have $a \in [0,1], c_i \in [0,1]$ and $c_i + c_j > 1-a$ for all $i$
and all $j\neq i$. The key observation is that the latter condition implies that
\begin{align*}
\P(T \leq 1) = \P\left(\sum_{i\in C} B_i \leq 1, \sum_{i\in D} B_i = 0\right)
\end{align*}
where 
$C :=\{i : c_i \leq 1 - a\}, D := \{i : c_i > 1 - a\}$.

If $\sum_{i\in C} \E B_i < 1$ and $C \neq \emptyset$, then we put $U = 1 + \sum_{i\in D} B_i$, 
noting that $U \in \cup_{m=1}^n \mathcal{U}_n$, giving
\begin{align*}
\P(U \leq 1) &= \P\left(\sum_{i\in D} B_i = 0\right) \\
&\geq \P\left(\sum_{i\in C} B_i \leq 1, \sum_{i\in D} B_i = 0\right) = \P(T \leq 1) \\
\text{and } \quad
\E U &= 1 + \sum_{i\in D} \E B_i \\
&\geq a + (1-a) \sum_{i\in C} \E B_i + \sum_{i\in D} \E B_i 
\ \text{ (as $\sum_{i\in C} \E B_i < 1$)}\\
& \geq a + \sum_{i\in C} c_i \E B_i + \sum_{i\in D} c_i \E B_i = \E T
\end{align*}
satisfying~(\ref{Ucondition}).

Finally, if $\sum_{i\in C} \E B_i \geq 1$ or $C = \emptyset$, then we put $U = \sum_{i=1}^n B_i$, noting that $U\in \mathcal{U}_n$ so that
\begin{align*}
\P(U \leq 1) &= \P\left(\sum_{i\in C} B_i + \sum_{i\in D} B_i \leq 1\right) \\
&\geq  \P\left(\sum_{i\in C} B_i \leq 1, \sum_{i\in D} B_i = 0\right) = \P(T \leq 1) \\
\text{and } \quad
\E U &= \sum_{i\in C} \E B_i + \sum_{i\in D} \E B_i \\
&\geq a + (1-a) \sum_{i\in C} \E B_i + \sum_{i\in D} \E B_i \\
& \geq a + \sum_{i\in C} c_i \E B_i + \sum_{i\in D} c_i \E B_i = \E T
\end{align*}
satisfying~(\ref{Ucondition}) and completing the proof. 
\end{proof}

\begin{lemma}
\label{lemmaU}
For any $U\in \mathcal{U}_n$ with $n\ge 1$, there exists a $V \in \cup_{m=1}^n \mathcal{V}_m$ such that
\begin{align}
\P(U\le 1) \le \P(V\le 1) \quad \text{ and } \quad \E U = \E V.
\end{align}
\end{lemma}

\begin{proof}
Let $U := \sum_{i=1}^n B_i$ where $B_i$ are Bernoulli random variables, 
$q_i := \E B_i$ and $q := (q_1, q_2, \dots, q_n)$. We have 
\begin{align}
\P(U\le 1) &= \P\left(\sum_{i=1}^n B_i =0\right) + \P\left(\sum_{i=1}^n B_i =1\right) \\
		&= \prod_{i=1}^n (1-q_i) + \sum_{j=1}^n q_j \prod_{i=1 : n, i \neq j} (1-q_i) 
=: L_n(q ) .
\end{align}

Consider maximizing $L_n(q)$ over $q \in \{ [0, 1]^n \mid \lambda = \sum_{i=1}^n q_i \}$
noting that maxima might lie on the interior with $q_i \in (0,1)$ for all $1\le i\le n$ or on the boundary with $q_i \in \{0, 1\}$ for some $i$.
Since $L_n(q)$ is a differentiable function of $q$, any critical point of $L_n(q)$ on the interior with
$q \in \{ (0, 1)^n \mid \lambda = \sum_{i=1}^n q_i \}$ must satisfy
\begin{align}
\label{LagrangeCondition}
\nabla_{q_k} \left(L_n(q) + \mu \sum_{i=1}^n q_i\right) = 0 \ \text{ for all $1\le k \le n$ }
\end{align}
for a suitable Lagrange multiplier $\mu$.
However, $L_n(q)$ is a symmetric linear function of each $q_k$. 
So if $n\ge 2$ then
any solution of equation~(\ref{LagrangeCondition}) must have
$q_k = q_l$ for all $k\neq l$ in $1, \dots, n$. 
Thus 
$q = \left( \frac{\lambda}{n}, \frac{\lambda}{n}, \dots, \frac{\lambda}{n}\right)$
for which $U$ corresponds to the random variable $V_{n,0} \sim \mathrm{binomial}\left(\frac{\lambda}{n}, n\right)$ which is in $\mathcal{V}_n$ and has
$\E V_{n,0} = \lambda$.

If $q_i = 1$ for some $i$ and $n\ge 2$ then the arithmetic-geometric mean inequality gives
\begin{align}
\label{Gbit}
L_n(q) &= \prod_{j=1 : n, j \neq i} (1-q_j) \le \left( 1 - \frac{\lambda - 1}{n-1} \right)^{n-1} .\end{align}
The right-hand side is $\P(V_{n,1} \le 1)$ for the random variable $V_{n,1} \sim
1+\mathrm{binomial}\left(\frac{\lambda-1}{n-1}, n-1\right)$
for which $V_{n,1} \in \mathcal{V}_n$ and $\E V_{n,1} = \lambda$.

If $q_i = 0$ for some $i$ and $n\ge 2$ then the definition of $L_n(q)$ gives
\begin{align}
\label{Lbit}
L_n(q) = L_{n-1}(q^i)
\text{ where $q^i := (q_1, \dots, q_{i-1}, q_{i+1}, \dots, q_n)$} .
\end{align}
However to have $q_i = 0$ for some $i$ we require that 
$\lambda = \sum_{j=1, j\ne i}^n q_j \le n-1$. 

In summary, if $q \in \{ [0, 1]^n \mid \lambda = \sum_{i=1}^n q_i \}$
and $n\ge 2$ then
\begin{align*}
L_n(q) &\le \begin{cases}
\max_{1\le i\le n} \{ L_{n-1}(q^i), \P(V_{n,0} \le 1), \P(V_{n,1} \le 1) \} & \text{if $0\le\lambda\le n-1$} \\
\max\{\P(V_{n,0} \le 1), \P(V_{n,1} \le 1)\} & \text{if $n-1<\lambda\le n$} ,
\end{cases} 
\end{align*}
and for $n=1$, consider the random variable $V_{1,1} := \mathrm{binomial}(\lambda, 1)$ for which $V_{1,1} \in \mathcal{V}_1$, $\P(U\le 1) = \P(V_{1,1} \le 1)$ and $\lambda = \E V_{1,1}$.
Thus 
\begin{align*}
\P(U\le 1) &\le \max_{\{ V \mid V\in \cup_{m=1}^n \mathcal{V}_m, \E V = \lambda \}} \P(V \le 1)  \end{align*}
which completes the proof. 
\end{proof}

\begin{lemma}
\label{lemmaV}
Let $H_n(\lambda) := \sup \{ \P(V\le 1) \mid V\in \mathcal{V}_n, \E V = \lambda \}$
with the convention that $\sup\emptyset = 0$. Then
\begin{align}
H_{n}(\lambda) &\le H_{n'}(\lambda') & \text{ for all $0 \le \lambda' \le \lambda$ and all $1\le n\le n'$} . 
\end{align}
\end{lemma}

\begin{proof}
The definition of $\mathcal{V}_n$ gives
\begin{align}
H_n(\lambda) &= \begin{cases}
1 & \text{if } 0\le \lambda \le 1 \\
\max\{F_n(\lambda), G_n(\lambda)\} & \text{if } 1 < \lambda < n \\
0 & \text{if } n \le \lambda
\end{cases} 
\intertext{where for $1\le \lambda\le n$}
F_n(\lambda) &:= 
\P\left(\mathrm{binomial}\left(\frac{\lambda}{n}, n\right) \le 1\right) 
= \left( 1 - \frac{\lambda}{n} \right)^n + \lambda \left( 1 - \frac{\lambda}{n}\right)^{n-1} \\
G_n(\lambda) &:= 
\P\left(1 + \mathrm{binomial}\left(\frac{\lambda-1}{n-1}, n-1\right) \le 1\right) 
= \left( 1 - \frac{\lambda - 1}{n-1} \right)^{n-1} ,
\end{align}
so let us collect some facts about $F_n(\lambda)$ and $G_n(\lambda)$.

First, set $x:=1-\lambda / n$ so we have $n = \lambda / (1-x)$ and
\begin{align}
\label{gDef}
\log F_{n}(\lambda) &= \log ( x^n + \lambda x^{n-1} ) 
= \frac{\lambda \log x}{1-x} + \log \left(1 + \frac{\lambda}{x} \right) =: g(x) . 
\end{align}
Now
\begin{align}
\label{uDef}
\frac{(1-x)^2}{\lambda} \nabla g(x) &=
 \log x + \frac{1-x}{x} - \frac{(1-x)^2}{x(x+\lambda)} =: u(x) 
\intertext{and } \nabla u(x) &= 
\frac{x-1}{x^2 (x+\lambda)^2} (x^2 + (\lambda-2) x + \lambda^2 - \lambda) .
\end{align}
Note that $\min_{x\in\mathbb{R}} (x^2 + (\lambda-2) x + \lambda^2 - \lambda) = (3 \lambda^2 - 4) / 4$. Thus if $\lambda \ge 2 / \sqrt{3}$ then
$\nabla u(x) \le 0$ for all $0<x\le 1$ and so $u(x) \ge u(1) = 0$.
Thus $\nabla g(x) \ge 0$, by (\ref{uDef}), so that 
$\log F_n(\lambda)$ is increasing in $n$, by (\ref{gDef}) and from the fact that
$n = \lambda / (1-x)$ is increasing in $x$ for fixed $\lambda$.
Hence
\begin{align}
\label{FIneq}
F_n(\lambda) \le F_{n+1}(\lambda) \ \text{ for all $2 / \sqrt{3}\le\lambda<n$ and $n\ge 1$.}
\end{align}

Second, Taylor expansion gives
\begin{align}
\log G_{n+1}(\lambda) = n \log\left(1 - \frac{\lambda}{n}\right) 
= - \lambda - \sum_{k=2}^\infty \frac{\lambda^k}{k n^{k-1}} 
\qquad \text{for $\left|\frac{\lambda}{n}\right| < 1$}
\end{align}
which is a non-decreasing function of $n$ for $\lambda \ge 0$. Thus
\begin{align}
\label{GIneq}
G_{n}(\lambda) \le G_{n+1}(\lambda) \ \text{ for all $0\le \lambda<n$ and $n\ge 1$}.
\end{align}

Third, considering the range of $\lambda$ for which $F_n(\lambda) \le G_n(\lambda)$ gives
\begin{align}
&& \left( 1 - \frac{\lambda}{n} \right)^n + \lambda \left( 1 - \frac{\lambda}{n}\right)^{n-1}
&\le\left( 1 - \frac{\lambda - 1}{n-1} \right)^{n-1} \\
&\Leftrightarrow& 1 - \frac{\lambda}{n} + \lambda 
&\le \left( \frac{1 - \frac{\lambda-1}{n-1}}{1 - \frac{\lambda}{n} } \right)^{n-1} 
= \left( \frac{n}{n-1} \right)^{n-1} \\
\label{FGCondition}
&\Leftrightarrow& \lambda&\le\left( \frac{n}{n-1} \right)^n - \frac{n}{n-1} . 
\end{align}
Applying the inequality $\log x \le x - 1$ to $x = \frac{n-1}{n}$ we see that
$n \log \frac{n-1}{n} \le -1$, hence
$\left( \frac{n}{n-1} \right)^n \ge e$. So for $n\ge 3$ we have 
\begin{align}
\label{LambdaN}
\left( \frac{n}{n-1} \right)^n - \frac{n}{n-1} \ge e - \frac{3}{2} > \frac{2}{\sqrt{3}} .
\end{align}
Additionally $G_2(\lambda)-F_2(\lambda) = (\lambda - 2)^2 / 4 \ge 0$ for all $\lambda\in\mathbb{R}$.
In conjunction with~(\ref{FGCondition}) and~(\ref{LambdaN}) this gives
\begin{align}
\label{FGIneq}
F_n(\lambda) \le G_n(\lambda) \text{ for all $\lambda < \frac{2}{\sqrt{3}}$ and $n\ge 2$.}
\end{align}

Now consider the function $H_n(\lambda)$.
The definition of $H_n(\lambda)$ gives 
\begin{align}
H_n(\lambda) &= H_{n+1}(\lambda) = 1 & \text{for all $0\le\lambda\le 1$ and $n\ge 1$}
\\
0 = H_n(\lambda) &\le H_{n+1}(\lambda) & \text{for all $n \le \lambda < n+1$ and $n\ge 1$} \\
H_n(\lambda) &= H_{n+1}(\lambda) = 0 & \text{for all $\lambda \ge n+1$ and $n\ge 1$.}
\end{align}
For all $1 < \lambda < \frac{2}{\sqrt{3}}$ and all $n\ge 2$, (\ref{GIneq}) and (\ref{FGIneq}) give
\begin{align}
H_n(\lambda) &= \max\{F_n(\lambda), G_n(\lambda)\} 
			= G_n(\lambda) \le G_{n+1}(\lambda) \\
			&\le \max\{F_{n+1}(\lambda), G_{n+1}(\lambda)\} 
		= H_{n+1}(\lambda) .
\intertext{For all $\frac{2}{\sqrt{3}} \le \lambda < n$ and all $n\ge 2$, (\ref{FIneq}) and (\ref{GIneq}) give}
H_n(\lambda) &= \max\{F_n(\lambda), G_n(\lambda)\} \\
		&\le \max\{F_{n+1}(\lambda), G_{n+1}(\lambda)\} = H_{n+1}(\lambda) .
\intertext{In summary}
\label{HIneq}
H_n(\lambda) &\le H_{n+1}(\lambda) \qquad \text{ for all $\lambda \ge 0$ and $n\ge 1$}
\end{align}
showing that $H_n(\lambda)$ is non-decreasing in $n$.

Finally, $\nabla_\lambda F_n(\lambda) = - \frac{n-1}{n} \left( 1 - \frac{\lambda}{n} \right)^{n-1} \le 0$ and $\nabla_\lambda G_n(\lambda) \le 0$, so $H_n(\lambda)$ is non-increasing in $\lambda$. This completes the proof. 
\end{proof}

\begin{proof}[Theorem~\ref{theorem1}]
By Lemmas~\ref{lemmaS}, ~\ref{lemmaT} and~\ref{lemmaU}, there
exist random sums $T\in\mathcal{T}_n$, $U\in\cup_{m=1}^n \mathcal{U}_m$ 
and $V\in\cup_{m=1}^n\mathcal{V}_m$ such that
\begin{align*}
\label{BindAll}
\P(S\le 1) \le \P(T\le 1)  \le \P(U\le 1)  \le \P(V\le 1) 
\text{ and }  \E S = \E T \le \E U = \E V .
\end{align*}
Say that $V\in \mathcal{V}_m$ for some $1\le m\le n$ and let
$\lambda_V := \E V$. Then $\lambda_V \ge \lambda =: \E S$ as just shown,
so Lemma \ref{lemmaV} gives
\begin{align*}
\P(V\le 1) &\le H_m(\lambda_V) \le H_n(\lambda) .
\end{align*}
Now, by definition of $H_n(\lambda)$, for $0\le \lambda \le n$ we have $H_n(\lambda) = \max\{F_n(\lambda), G_n(\lambda) \}$ where
$F_n(\lambda) := \left(1 + \lambda -\frac{\lambda}{n}\right)\left( 1 - \frac{\lambda}{n} \right)^{n-1}$ and $G_n(\lambda) := \left( 1 - \frac{\lambda - 1}{n-1} \right)^{n-1},$ so that
\begin{align}
\P(S\le 1) \le \max\{ F_n(\lambda), G_n(\lambda) \}
\end{align}
which proves~(\ref{theorem1_1}). Furthermore, Lemma~\ref{lemmaV} gives
\begin{align}
H_n(\lambda) \le 
\lim_{m\rightarrow\infty} \max\{  F_m(\lambda), G_m(\lambda) \}
= \max\{ 1+\lambda, e \} e^{-\lambda}
\end{align}
which completes the proof. 
\end{proof}

\section{Application}
If we wish to bound the expectation of a random sum,
then Theorem~\ref{theorem1} can be conveniently rearranged as follows.
\begin{corollary}
\label{Corollary1}
Suppose that $S = \sum_{i=1}^n X_i$ is a sum of independent random variables with
$\mathbb{P}(0\leq X_i \leq 1)=1$ for $1\leq i\leq n$. 
Then for $r = 0.841405\dots$, we have
\begin{align}
\P (S \leq 1) \leq e^{1 - r \E S} \text{ or equivalently }
\E S \leq \frac{1}{r} \left(1 - \log \mathbb{P}(S\leq 1) \right) .
\end{align}
\end{corollary}

\begin{proof}
We work with the right-hand side of Theorem~\ref{theorem1} to find the smallest $a$ such that 
\begin{align*}
\max\{e, 1+m \} e^{-m} \leq e^{1-m+am}
\end{align*}
for all $m\ge 0$,
or equivalently, such that 
\begin{align*}
\log(1+m) - a m \leq 1.
\end{align*}
For fixed $a\geq 0$ the left-hand side is concave with a unique maximum at $m = \frac{1}{a}-1$.
Substituting this $m$, we require that 
\begin{align*}
a - \log a \leq 2.
\end{align*}
Now the function $a - \log a$ is decreasing for $a \leq 1$, thus we require that $a \geq a_0$ where $a_0$ 
is the root of $a_0 = e^{a_0 - 2}$ having $a_0 \leq 1$.
A fixed point method yields the solution $a_0 = 0.158594\dots = 1-r$. 
\end{proof}

\section*{Acknowledgements}
The author thanks the anonymous reviewer, Christophe Leuridan, Bin Yu, Nicol{\` o} Cesa-Bianchi, Onno Zoeter and Shengbo Guo for helpful comments and discussions.

The final version of this publication is (or will be) available at springerlink.com in the Journal of Theoretical Probability.

\bibliographystyle{spmpsci}

\end{document}